\documentclass[preprint,12pt]{elsarticle}

\usepackage{graphicx}
\usepackage{amsmath}
\usepackage{amsthm}
\usepackage{dsfont}
\usepackage{xcolor}
\usepackage{amsfonts}
\usepackage{amssymb}
\usepackage{epsfig,psfrag,pstricks}
\usepackage{subfigure}
\usepackage{tikz}
\usepackage{tkz-berge}
\usetikzlibrary{shapes,arrows}
\DeclareMathOperator{\im}{im}

\DeclareMathOperator{\spa}{span}
\DeclareMathOperator{\sat}{sat}

\newtheorem{theorem}{Theorem}
\newtheorem{lemma}[theorem]{Lemma}
\newtheorem{assumption}[theorem]{Assumption}
\newtheorem{definition}[theorem]{Definition}

\newtheorem{corollary}[theorem]{Corollary}
\newtheorem{exmp}{Example}[section]
\newtheorem{remark}[theorem]{Remark}

\journal{Systems \& Control Letters}

\begin{document}

\begin{frontmatter}


\title{Load balancing of dynamical distribution
networks with flow constraints and unknown in/outflows}
\fntext[1]{The work of the second author is supported by the Chinese Science
Council (CSC).}
\fntext[2]{The research of the first author leading to these
results has received funding from the EU 7th Framework Programme
[FP7/2007-2013] under grant agreement no. 257462 HYCON2 Network of
Excellence.}
\author{J. Wei\fnref{1}}
\ead{J. Wei@rug.nl}
\author{A.J. van der Schaft\fnref{2}}
\ead{A.J.van.der.Schaft@rug.nl}
\address{Johann Bernoulli Institute for Mathematics and Computer
Science, University of Groningen, PO Box 407, 9700 AK, the
Netherlands.}






\begin{abstract}
We consider a basic model of a dynamical distribution network, modeled as a directed
graph with storage variables corresponding to every vertex and flow inputs
corresponding to every edge, subject to unknown but constant
inflows and outflows. As a preparatory result it is  shown how a distributed
proportional-integral controller structure, associating
with every edge of the graph a controller state, will regulate the state
variables of the vertices, irrespective of the unknown constant inflows and
outflows, in the sense that the storage variables converge to the same value
(load balancing or consensus). This will be proved by identifying the closed-loop system as a
port-Hamiltonian system, and modifying the Hamiltonian function into a Lyapunov function, dependent
on the value of the vector of constant inflows and outflows. In the main part of
the paper the same problem will be addressed for the case that the input flow
variables are {\it constrained} to take value in an interval. We will derive
sufficient and necessary conditions for load balancing, which only depend on the
structure of the network in relation with the flow constraints.
\end{abstract}

\begin{keyword}

PI controllers \sep flow constraints \sep directed graphs \sep port-Hamiltonian
systems \sep consensus algorithms \sep Lyapunov stability
\end{keyword}

\end{frontmatter}


\section{Introduction}
In this paper we study a basic model for the dynamics of a distribution
network. Identifying the network with a directed graph we associate with every
vertex of the graph a state variable corresponding to {\it storage}, and with
every edge a control input variable corresponding to {\it flow}, which is constrained
to take value in a given closed interval. Furthermore, some of the vertices
serve as terminals where an unknown but constant flow may enter or leave the
network in such a way that the total sum of inflows and outflows is equal to
zero. The control problem to be studied is to derive necessary and sufficient
conditions for a distributed control structure (the control input corresponding to a
given edge only depending on the difference of the state variables of the
adjacent vertices) which will ensure that the state variables associated to all
vertices will converge to the same value equal to the average of the initial condition,
irrespective of the values of the constant unknown inflows and outflows.

The structure of the paper is as follows. Some preliminaries and notations will
be given in Section 2. In Section 3 we will show how in the absence of
constraints on the flow input variables a distributed proportional-integral (PI) controller structure,
associating with every edge of the graph a controller state, will solve the
problem if and only if the graph is weakly connected. This will
be shown by identifying the closed-loop system as a port-Hamiltonian system,
with state variables associated both to the vertices and the edges of the graph,
in line with the general definition of port-Hamiltonian systems on graphs
\cite{schaftSIAM, schaftCDC08, schaftBosgrabook,
schaftNECSYS10}; see also \cite{allgower11,Mesbahi11}. The proof of asymptotic
load balancing will be given by modifying, depending on the vector
of constant inflows and outflows, the total Hamiltonian function into a Lyapunov
function. In examples the obtained PI-controller often has a clear physical
interpretation, emulating the physical action of adding energy storage and
damping to the edges.

The main contribution of the paper resides in Sections 4 and 5, where the same 
problem is addressed for the case of {\it constraints} on the flow input
variables. In Section 4 it will be shown that in the case of
{\it zero} inflow and outflow the state variables of the vertices converge to the same value if and only if
the network is strongly connected. This will be shown by constructing a $C^1$
Lyapunov function based on the total Hamiltonian and the constraint values. This
same construction will be extended in Section 5
to the case of non-zero inflows and outflows, leading to the result that in this
case asymptotic load balancing is reached if and only the graph is not only strongly
connected but also {\it balanced}. Finally, Section 6 contains the conclusions.

Some preliminary results, in particular concerning Section 3, have been already reported before in \cite{Wei2012}.

\section{Preliminaries and notations}

First we recall some standard definitions regarding directed graphs,
as can be found e.g. in \cite{Bollobas98}. A \textit{directed graph}
$\mathcal{G}$ consists of a finite set $\mathcal{V}$ of \textit{vertices}
and a finite set $\mathcal{E}$ of \textit{edges}, together
with a mapping from $\mathcal{E}$ to the set of ordered pairs of
$\mathcal{V}$, where no self-loops are allowed. Thus to any edge
$e\in\mathcal{E}$ there corresponds an ordered pair
$(v,w)\in\mathcal{V}\times\mathcal{V}$
(with $v\not=w$), representing the tail vertex $v$ and the head
vertex $w$ of this edge.

A directed graph is completely specified by its \textit{incidence
matrix} $B$, which is an $n\times m$ matrix, $n$ being the
number of vertices and $m$ being the number of edges, with $(i,j)^{\text{th}}$
element equal to $1$ if the $j^{\text{th}}$ edge is towards vertex
$i$, and equal to $-1$ if the $j^{\text{th}}$ edge is originating from
vertex $i$, and $0$ otherwise. Since we will only consider directed
graphs in this paper `graph' will throughout mean `directed graph'
in the sequel.
A directed graph is {\it strongly connected} if it is
possible to reach any vertex starting from any other vertex by traversing edges 
following their directions. A directed graph is called {\it weakly connected}
if it is possible to reach any vertex from every other vertex using the edges
{\it not} taking into account their direction. A graph is weakly connected if
and only if $\ker B^T = \spa \mathds{1}_n$. Here $\mathds{1}_n$ denotes the
$n$-dimensional vector with all elements equal to $1$. A graph that is not
weakly connected falls apart into a number of weakly connected subgraphs, called
the connected components. The number of connected components is equal to $\dim
\ker B^T$.
For each vertex, the number of incoming edges is called the {\it in-degree} of
the vertex and the number of outgoing edges its out-degree. A graph is called
{\it balanced} if and only if the in-degree and out-degree of every vertex are
equal. A graph is balanced if and only if $\mathds{1}_n \in \ker B$.

Given a graph, we define its \textit{vertex space} as the vector space of all
functions from $\mathcal{V}$ to some linear space $\mathcal{R}$. In the rest of
this paper we will take for simplicity $\mathcal{R}=\mathbb{R}$, in which case
the vertex space can be identified with $\mathbb{R}^{n}$. Similarly, we define
its \textit{edge space} as the
vector space of all functions from $\mathcal{E}$ to $\mathcal{R} = \mathbb{R}$,
which can be identified with $\mathbb{R}^{m}$. In this way, the incidence matrix
$B$ of the graph can be also regarded as the matrix representation of a linear
map from the edge space $\mathbb{R}^m$ to the vertex space $\mathbb{R}^n$.

\noindent
{\bf Notation}: For $a,b\in\mathbb{R}^m$ the notation $a \leqslant b$ will
denote elementwise inequality $a_i \leq b_i,\,i=1,\ldots,m$. For $a_i <
b_i,\,i=1,\ldots,m$ the multidimensional
saturation function
$\sat(x\,;a,b): \mathbb{R}^m\rightarrow\mathbb{R}^m$ is defined as
\begin{equation}
\sat(x\,;a,b)_i  = \left\{ \begin{array}{ll}
a_i & \textrm{if $x_i\leq a_i,$}\\
x_i & \textrm{if $a_i<x_i<b_i,$}\\
b_i & \textrm{if $x_i\geq b_i$},
\end{array}
\, i=1,\ldots,m. \right.
\end{equation}

\section{A dynamic network model with PI controller}
Let us consider the following dynamical system defined on the graph
$\mathcal{G}$
\begin{equation}\label{system}
\begin{array}{rcl}
\dot{x} & = & Bu, \quad x \in \mathbb{R}^n, u \in \mathbb{R}^m \\[2mm]
y & = & B^T \frac{\partial H}{\partial x}(x), \quad y \in \mathbb{R}^m,
\end{array}
\end{equation}
where $H: \mathbb{R}^n \to \mathbb{R}$ is any differentiable function, and $\frac{\partial H}{\partial
x}(x)$ denotes the column vector of partial derivatives of $H$. Here the
$i^{\text{th}}$
element $x_i$ of the state vector $x$ is the state variable
associated to the $i^{\text{th}}$ vertex, while $u_j$ is a flow input variable
associated to the
$j^{\text{th}}$ edge of the graph. System (\ref{system}) defines a
port-Hamiltonian system (\cite{vanderschaftmaschkearchive,
vanderschaftbook}), satisfying the energy-balance
\begin{equation}
\frac{d}{dt}H = u^Ty.
\end{equation}
Note that geometrically its state space is the vertex space, its input space is
the edge space, while its output space is the dual of the
edge space.

\begin{exmp}[Hydraulic network]
Consider a hydraulic network, modeled as a directed graph with vertices (nodes)
corresponding to reservoirs, and edges (branches) corresponding to pipes.
Let $x_i$ be the stored water at vertex $i$, and $u_j$ the flow through edge
$j$. Then the mass-balance of the network is summarized in
\begin{equation}
\dot{x} = Bu,
\end{equation}
where $B$ is the incidence matrix of the graph. Let furthermore $H(x)$ denote
the stored energy in the reservoirs (e.g., gravitational energy). Then $P_i :=
\frac{\partial H}{\partial x_i}(x), i=1, \ldots, n,$ are the {\it pressures} at
the vertices, and the output vector $y= B^T \frac{\partial H}{\partial x}(x)$ is
the vector whose $j^{\text{th}}$ element is the pressure {\it difference}
$P_i - P_k$
across the $j^{\text{th}}$ edge linking vertex $k$ to vertex $i$.
\end{exmp}

As a next step we will extend the dynamical system (\ref{system}) with a vector
$d$ of {\it inflows and outflows}
\begin{equation}\label{system1}
\begin{array}{rcl}
\dot{x} & = & Bu + Ed, \quad x \in \mathbb{R}^n, u \in \mathbb{R}^m, \quad d \in
\mathbb{R}^k \\[2mm]
y & = & B^T \frac{\partial H}{\partial x}(x), \quad y \in \mathbb{R}^m,
\end{array}
\end{equation}
with $E$ an $n \times k$ matrix whose columns consist of exactly one entry equal
to $1$ (inflow) or $-1$ (outflow), while the rest of the elements is zero. Thus
$E$ specifies the $k$ terminal vertices where flows can enter or leave the
network.

In this paper we will regard $d$ as a vector of constant {\it disturbances},
and we want to investigate control schemes which
ensure asymptotic load balancing of the state vector $x$ irrespective of the
(unknown) disturbance $d$. The
simplest control possibility is to apply a proportional output feedback
\begin{equation}\label{Pcontroller}
u =-Ry = -RB^T\frac{\partial H}{\partial x}(x),
\end{equation}
where $R$ is a diagonal matrix with strictly positive diagonal elements
$r_1,\ldots,r_m$. Note that this defines a {\it decentralized} control scheme if
$H$ is of the form $H(x)= H_1(x_1) + \ldots + H_n(x_n)$, in which case the
$i^{\text{th}}$ input is given as $r_i$ times the difference of the component of
$\frac{\partial H}{\partial x}(x)$ corresponding to the head vertex of the
$i^{\text{th}}$ edge and the component of $\frac{\partial H}{\partial x}(x)$
corresponding to its tail vertex.
This control scheme leads to the closed-loop system
\begin{equation}\label{closedloop1}
\dot{x} = -BRB^T \frac{\partial H}{\partial x}(x) + Ed.
\end{equation}
In case of {\it zero} in/outflows $d=0$ this implies the energy-balance
\begin{equation}
\frac{d}{dt}H = - \frac{\partial^T H}{\partial x}(x)BRB^T \frac{\partial
H}{\partial x}(x) \leq 0.
\end{equation}
Hence if $H$ is radially unbounded it follows that the system trajectories of the closed-loop system
(\ref{closedloop1}) will
converge to the set
\begin{equation}
\mathcal{E} := \{x \mid B^T \frac{\partial H}{\partial x}(x)=0 \}.
\end{equation}
and thus to the load balancing set
\[
\mathcal{E} = \{x \mid \frac{\partial H}{\partial x}(x) = \alpha \mathds{1},
\alpha \in \mathbb{R} \}.
\]
if and only if $\ker B^T = \spa\{\mathds{1}\}$, or equivalently
\cite{Bollobas98}, if and only if the graph is {\it weakly connected}.

In particular, for the standard Hamiltonian $H(x) = \frac{1}{2} \| x \|^2$ this
means that the state variables
$x_i(t), i=1, \ldots,n,$ converge to a common value $\alpha$ as $t \to \infty$. Since $\frac{d}{dt} \mathds{1}^Tx(t)=0$ it follows that this common value is given as $\alpha = \frac{1}{n} \sum_{i=1}^n x_i(0)$.

\medskip
For $d \neq 0$ proportional control $u=
-Ry$ will not be sufficient to reach load balancing, since the disturbance $d$ can only be attenuated at the
expense of increasing the gains in the matrix $R$. Hence we
consider {\it proportional-integral} (PI) control given by the dynamic output
feedback
\begin{equation}\label{PI}
\begin{array}{rcl}
\dot{x}_c & = & y ,\\[2mm]
u & = &-Ry - \frac{\partial H_c}{\partial x_c}(x_c),
\end{array}
\end{equation}
where $H_c(x_c)$ denotes the storage function (energy) of the controller. Note
that this PI controller is of the same decentralized nature as the static output
feedback $u=-Ry$.

The $j^{\text{th}}$ element of the controller state $x_c$ can be regarded as an
additional state variable corresponding to the $j^{\text{th}}$ edge. Thus $x_c \in
\mathbb{R}^m$, the edge space of the network. The closed-loop system resulting
from the PI control (\ref{PI}) is given as
\begin{equation}\label{closedloop}
\begin{bmatrix} \dot{x} \\[2mm] \dot{x}_c \end{bmatrix} =
\begin{bmatrix} -BRB^T & -B \\[2mm] B^T & 0 \end{bmatrix}
\begin{bmatrix} \frac{\partial H}{\partial x}(x) \\[2mm] \frac{\partial
H_c}{\partial x_c}(x_c) \end{bmatrix} +
\begin{bmatrix} E \\[2mm] 0 \end{bmatrix} d,
\end{equation}
This is again a port-Hamiltonian system\footnote{See
\cite{schaftSIAM, schaftCDC08, schaftBosgrabook, schaftNECSYS10}
for a general definition of port-Hamiltonian systems on graphs.},
with total
Hamiltonian $H_{\mathrm{tot}}(x,x_c)\\ := H(x) + H_c(x_c)$, and satisfying
the
energy-balance
\begin{equation}\label{Lyapunov}
\frac{d}{dt}  H_{\mathrm{tot}}= - \frac{\partial^T H}{\partial x}(x)BRB^T
\frac{\partial H}{\partial x}(x) + \frac{\partial^T H}{\partial x}(x)Ed
\end{equation}
Consider now a constant disturbance $\bar{d}$ for which there exists a {\it matching}
controller state $\bar{x}_c$, i.e.,
\begin{equation}\label{matching}
E \bar{d} = B\frac{\partial H_c}{\partial x_c}(\bar{x}_c).
\end{equation}
This allows us to modify the total Hamiltonian $H_{\mathrm{tot}}(x,x_c)$ into\footnote{This function was introduced for passive systems with constant inputs in \cite{jaya1}.}
\begin{equation}
V_{\bar{d}}(x,x_c) := H(x) + H_c(x_c) - \frac{\partial^T H_c}{\partial
x_c}(\bar{x}_c)(x_c - \bar{x}_c) - H_c(\bar{x}_c),
\end{equation}
which will serve as a candidate Lyapunov function; leading to the
following theorem.
\begin{theorem}
Consider the system (\ref{system1}) on the graph $\mathcal{G}$ in closed loop with the PI-controller (\ref{PI}).
Let the constant disturbance $\bar{d}$ be such that there exists a $\bar{x}_c$
satisfying the matching equation (\ref{matching}). Assume that
$V_{\bar{d}}(x,x_c)$ is radially unbounded. Then the trajectories of the
closed-loop system (\ref{closedloop}) will converge to an element of the load balancing set
\begin{equation}
\mathcal{E}_{\mathrm{tot}} = \{ (x,x_c) \mid \frac{\partial H}{\partial x}(x) =
\alpha \mathds{1}, \, \alpha \in \mathbb{R}, \, B\frac{\partial H_c}{\partial
x_c}(x_c) = E\bar{d}\, \}.
\end{equation}
if and only if $\mathcal{G}$ is weakly connected.
\end{theorem}
\begin{proof} Suppose $\mathcal{G}$ is weakly connected. By (\ref{Lyapunov}) for $d=\bar{d}$ we obtain, making use of
(\ref{matching}),
\begin{equation}\label{Lyapunov1}
\begin{aligned}
\frac{d}{dt} V_{\bar{d}} =& - \frac{\partial^T H}{\partial x}(x)BRB^T
\frac{\partial H}{\partial x}(x) + \frac{\partial^T H}{\partial x}(x)E\bar{d}-
\\[3mm]
& \frac{\partial^T H_c}{\partial x_c}(\bar{x}_c)B^T\frac{\partial
H}{\partial x}(x) \\[3mm]
 =& - \frac{\partial^T H}{\partial x}(x)BRB^T \frac{\partial H}{\partial x}(x)
\leq 0.
\end{aligned}
\end{equation}
Hence by LaSalle's invariance principle the system trajectories converge to the
largest invariant set contained in
\[
\{(x,x_c) \mid B^T \frac{\partial H}{\partial x}(x) =0 \}.
\]
Substitution of $B^T \frac{\partial H}{\partial x}(x) =0$ in the closed-loop
system equations (\ref{closedloop}) yields $x_c$ constant and $-B\frac{\partial
H_c}{\partial x_c}(x_c) + E\bar{d}=0$. Since the graph is weakly connected $B^T \frac{\partial H}{\partial x}(x) =0$ implies $\frac{\partial H}{\partial x}(x) = \alpha \mathds{1}$. If the graph is not weakly connected then the above analysis will hold on every connected component, but the common value $\alpha$ will be different for different components.
\end{proof}
\begin{corollary}
If $\ker B = 0$, which is equivalent (\cite{Bollobas98}) to the graph having no
{\it cycles}, then for every $\bar{d}$ there exists a unique $\bar{x}_c$
satisfying (\ref{matching}), and convergence is towards the set
$\mathcal{E}_{\mathrm{tot}} = \{ (x, \bar{x}_c)
\mid \frac{\partial H}{\partial x}(x) = \alpha \mathds{1}, \alpha \in
\mathbb{R}, \, x_c=\bar{x}_c \}$.
\end{corollary}
\begin{corollary}
In case of the standard quadratic Hamiltonians $H(x) = \frac{1}{2} \| x \|^2$,
$H_c(x_c)=\frac{1}{2} \| x_c \|^2$ there exists for every $\bar{d}$
a controller state $\bar{x}_c$ such that (\ref{matching}) holds if and only if
\begin{equation}\label{matching1}
\im E \subset \im B.
\end{equation}
Furthermore, in this case $V_{\bar{d}}$ equals the radially unbounded function $\frac{1}{2}
\| x \|^2 + \frac{1}{2} \| x_c - \bar{x}_c \|^2$, while convergence will be
towards the load balancing set $\mathcal{E}_{\mathrm{tot}} = \{ (x,x_c) \mid x = \alpha
\mathds{1}, \alpha \in \mathbb{R},\, Bx_c = E\bar{d}\}$.
\end{corollary}

A necessary (and in case the graph is weakly connected necessary {\it and}
sufficient) condition for the inclusion $\im E \subset \im B $ is that $\mathds{1}^TE =
0$. In its turn $\mathds{1}^TE =
0$ is equivalent to the fact that for every $\bar{d}$ the total inflow into the network equals to the
total outflow). The condition $\mathds{1}^TE = 0$ also implies
\begin{equation}
\mathds{1}^T\dot{x} = -\mathds{1}^TBRB^T\frac{\partial H}{\partial x}(x) +
\mathds{1}^TE\bar{d}=0,
\end{equation}
implying (as in the case $d=0$) that $\mathds{1}^Tx$ is a {\it conserved quantity} for the closed-loop
system (\ref{closedloop}). In particular it
follows that the limit value $\lim_{t \to \infty}x(t) \in \spa\{ \mathds{1}\}$
is
determined by the initial condition $x(0)$.

\begin{exmp}[Hydraulic network continued]
The proportional part $u=-Ry$ of the controller corresponds to adding {\it damping} to the dynamics (proportional to
the pressure differences along the edges). The integral part of the controller
has the interpretation of adding {\it compressibility} to the hydraulic network
dynamics. Using this emulated compressibility, the PI-controller is able to
regulate the hydraulic network to a load balancing situation where all pressures
$P_i$
are equal, irrespective of the constant inflow and outflow $\bar{d}$ satisfying
the matching condition (\ref{matching}). Note that for the Hamiltonian $H(x)=
\frac{1}{2}\|x\|^2$ the pressures $P_i$ are equal to each other if and only if the water
levels $x_i$ are equal.
\end{exmp}

\section{Constrained flows: the case without in/out flows}
In many cases of interest, the elements of the vector of flow inputs $u \in \mathbb{R}^m$
corresponding to the edges of the graph will be {\it constrained}, that is
\begin{equation}
 u \in\mathcal{U}:=\{u\in\mathbb{R}^m\mid u^-\leqslant u\leqslant u^+\}
\end{equation}
for certain vectors $u^-$ and $ u^+$ satisfying $u^-_i\leqslant0 \leqslant
u^+_i, i=1,\ldots,m$ (throughout $\leqslant $ denotes element-wise inequality).
This leads to the following constrained version\footnote{See also
\cite{Blanchini00} for a related problem setting where a constrained
version of the proportional controller (\ref{Pcontroller}) is considered.} of the PI
controller (\ref{PI}) given in the previous section
\begin{equation}\label{PIconstrained}
\begin{array}{rcl}
\dot{x}_c & = & y ,\\[2mm]
u & = &\sat\big(-Ry - \frac{\partial H_c}{\partial x_c}(x_c)\,;u^-,u^+\big)
\end{array}
\end{equation}
Throughout this paper we make the following assumption on the flow constraints.
\begin{assumption}
\begin{equation}
u^-_i \leqslant 0, \quad u^+_i \geqslant 0, \quad u^-_i <  u^+_i, \, i=1, \ldots,m
\end{equation}
\end{assumption}
It is important to note that we may change the {\it orientation} of some of the
edges of the graph at will; replacing the corresponding columns $b_i$ of the
incidence matrix $B$ by $-b_i$. Noting the identity
$\sat(-x\,;u_i^-,u_i^+)=-\sat(x\,;-u_i^+,-u_i^-)$ this implies that we may
assume without loss of generality that the orientation of the graph is chosen
such that
\begin{equation}\label{assumption}
u^-_i \leqslant 0 < u^+_i, \quad i=1,\ldots,m
\end{equation}
{\bf This will be assumed throughout the rest of the paper}. In general, we
will say that any orientation of the graph is {\it compatible} with the flow
constraints if (\ref{assumption}) holds. If the $j$-th edge is such that $u^-_j = 0$ then we will call this edge an {\it uni-directional} edge, while if $u^-_j < 0$ then the edge is called a {\it bi-directional} edge.

In this section we will first analyze the closed-loop system for the constrained
PI-controller under the simplifying assumption of {\it zero inflow and outflow}
($d=0$). In the next section we will deal with the general case.
Furthermore, for simplicity of exposition we consider throughout the rest of this paper the standard
Hamiltonian $H_c(x_c) = \frac{1}{2} \| x_c \|^2$ for the constrained PI
controller and the identity gain matrix $R=I$, while we also throughout assume
that the Hessian matrix of
Hamiltonian $H(x)$ is positive definite for any $x$. Thus
we consider the closed-loop system
\begin{equation}\label{closedloop-sat}
\begin{array}{rcl}
\dot{x} & = & B \sat\big(-B^T\frac{\partial H}{\partial
x}(x)-x_c\,;u^-,u^+\big), \\[2mm]
\dot{x}_c & = & B^T\frac{\partial H}{\partial x}(x).
\end{array}
\end{equation}
In order to state the main theorem of this section we need one more definition 
concerning strong connectedness with respect to flow constraints.
\begin{definition}
Consider the directed graph $\mathcal{G}$ together with the constraint values 
$u^-,u^+$ satisfying (\ref{assumption}). Then we will call the graph strongly
connected {\it with respect to the flow constraints} $u^- \leqslant u \leqslant
u^+$ if the following holds: for every two vertices $v_1, v_2$ there exists an
orientation of the graph compatible with the flow constraints\footnote{Note that
for different pairs of vertices we may need different orientations compatible
with the flow constraints. Thus the definition of strong connectedness
with respect to the flow constraints is {\it weaker} than the existence of an
orientation of the graph compatible with the flow constraints in which the graph
is strongly connected.} and a directed path (directed with respect to this
orientation) from $v_1$ to $v_2$.
\end{definition}

\begin{theorem}\label{th:zerodisturbance}
Consider the closed-loop system $(\ref{closedloop-sat})$ on a graph
$\mathcal{G}$ with flow constraints $u^- \leqslant u \leqslant u^+$ satisfying
(\ref{assumption}). Then its solutions converge to the load balancing set
\begin{equation}\label{set-zerodist}
 \mathcal{E}_{\mathrm{tot}} = \{ (x,x_c) \mid \frac{\partial H}{\partial x}(x) =
\alpha \mathds{1}_n, \, B\sat(-x_c\,;u^-,u^+) = 0 \}
\end{equation}
if and only if the graph is strongly connected with respect to the flow
constraints.
\end{theorem}
\begin{proof}
{\it Sufficiency}:
Consider the  Lyapunov
function given by
\begin{equation}\label{sat-Lyapunov}
V(x,x_c)=\mathds{1}^{T}_m S\big(-B^{T}\frac{\partial H}{\partial
x}(x)-x_c\,;u^-,u^+\big)+H(x),
\end{equation}
with
\begin{equation}
S(x\,;u^-,u^+)_i:=\int_0^{x_i} \sat(y\,;u^-_i,u^+_i)dy.
\end{equation}
It can be easily verified that $V$ is positive
definitive, radially unbounded and $C^1$. Its time-derivative is given as
\begin{equation}
\begin{aligned}
\dot{V} & =  -\sat^{T}\big(-B^{T}\frac{\partial H}{\partial
x}(x)-x_c;u^-,u^+\big)B^{T}B\sat\big(-B^{T}\frac{\partial H}{\partial
x}(x)-x_c\,;u^-,u^+\big) \\
& \leqslant  0.
\end{aligned}
\end{equation}
By LaSalle's invariance principle, all trajectories will converge to the largest
invariant set, denoted as $\mathcal{I}$, contained in $\mathcal{K}=\{
(x,x_c)|\, B\sat\big(-B^T\frac{\partial H}{\partial
x}(x)-x_c\,;u^-,u^+\big)=0 \}$. Whenever $x\in\mathcal{K}$ it
follows
that $\dot{x}=0$ and thus $x(t)=\nu$ for some constant vector $\nu$.
Hence, since $\dot{x}_c=B^T\frac{\partial H}{\partial x}(x)$, it follows that
$x_c(t)=B^T\frac{\partial
H}{\partial x}(\nu) t+x_c(0)$.

Suppose now that $B^T\frac{\partial H}{\partial x}(\nu) \neq 0$. Then
for $t$ large enough
\begin{equation}
\begin{aligned}
0 & = \frac{\partial^T H}{\partial x}(\nu)B\sat\big(-B^T\frac{\partial
H}{\partial
x}(\nu)-B^T\frac{\partial H}{\partial x}(\nu) t-x_c(0),u^-,u^+\big) \\
&
=\sum_{i=1}^{m}\big(B^T\frac{\partial
H}{\partial x}(\nu)\big)_ic_i ,
\end{aligned}
\end{equation}
where
\begin{equation}
c_i  = \left\{ \begin{array}{ll}
u^-_i & \textrm{if $\big(B^T\frac{\partial
H}{\partial x}(\nu)\big)_i>0$},\\
u^+_i & \textrm{if $\big(B^T\frac{\partial
H}{\partial x}(\nu)\big)_i<0$}.
\end{array} \right.
\end{equation}
Hence , in view of $u^-_i\leqslant0<u^+_i,$ we have $\big(B^T\frac{\partial
H}{\partial x}(\nu)\big)_i\geqslant0$, for $i=1,\ldots,m$.
However since the graph is strongly connected with respect to the flow
constraints, if
$\big(B^T\frac{\partial
H}{\partial x}(\nu)\big)_i>0,$ then there exists $j$ such that
$\big(B^T\frac{\partial
H}{\partial x}(\nu)\big)_j<0$. This yields a
contradiction. We conclude that $B^T\frac{\partial
H}{\partial x}(\nu)=0$, which implies $\frac{\partial
H}{\partial x}(\nu)=\alpha\mathds{1}_n$, and thus all trajectories converge to
$\mathcal{E}_{\mathrm{tot}}$.

{\it Necessity}:
Assume without loss of generality that the graph is weakly connected. (Otherwise
the same analysis can be performed on every connected component.)
If the graph is not strongly connected with respect to the flow constraints then there is a pair of vertices $v_i, v_j$ for which there exists a compatible orientation and a directed path from $v_i$ to
$v_j$, but not a compatible orientation and directed path from $v_j$ to $v_i$. In other words, there can be
positive flow from $v_i$ to $v_j$, but not vice versa. Then for suitable
initial condition, $\frac{\partial H}{\partial x_i} (x(t) < \frac{\partial
H}{\partial
x_j}(x(t) $ for all $t\geqslant0,$ and thus there is no convergence to
$\mathcal{E}_{\mathrm{tot}}$.
\end{proof}

\begin{remark}
Note that for $a_i \rightarrow -\infty,\,b_i\rightarrow \infty,$ the Lyapunov
function
$(\ref{sat-Lyapunov})$ tends to the function
$H(x)+\frac{1}{2}\|B^T\frac{\partial H}{\partial x}(x)+x_c\|^2$, which is
different from the Lyapunov function $H(x)+\frac{1}{2}\|x_c\|^2$ used in the
previous section.
\end{remark}
In the special case that the flow constraints are such that {\it all} the
flows $u_i$ can follow both directions, we obtain the following corollary.
\begin{corollary}\label{twoside-sat-cor}
For a network with constraint intervals $[u^-_i,u^+_i]$ with $u^-_i<0<u^+_i, i=1,\ldots,m,$ the trajectories of the
closed-loop system $(\ref{closedloop-sat})$ will converge to the set $(\ref{set-zerodist})$
if and only if the network is {\it weakly connected}.
\end{corollary}
\begin{proof}
In this case (since all the edges are bi-directional) weak connectedness is
equivalent to strong connectedness with respect to the flow constraints. If the
graph is not weakly connected then the components of $\frac{\partial H}{\partial
x}$ will only converge to a common value on every connected component.
\end{proof}

\section{Nonzero inflows and outflows}

In this section we deal with the general case of nonzero (but constant) inflows
and
outflows $\bar{d}$. Thus we consider the closed-loop system
\begin{equation}\label{closedloop-sat-disturb}
\begin{aligned}
\dot{x} & =  B\sat\big(-B^T\frac{\partial H}{\partial
x}(x)-x_c\,;u^-,u^+\big)+E\bar{d},
\\[2mm]
\dot{x}_c & =  B^T\frac{\partial H}{\partial x}(x),
\end{aligned}
\end{equation}
with $\im E\subset \im B.$

In order for the system to reach consensus, we need to impose conditions on the
magnitude of the in/outflows $\bar{d}$.
\begin{definition}
Given the constraint values $u^- < u^+$ the {\it permission set}
$\mathcal{P}(u^-,u^+)$ are
defined as
\[
\mathcal{P}_1(u^-,u^+)\times\mathcal{P}_2(u^-,u^+)\cdots\times\mathcal{P}_m(u^-
,u^+)\]
where the intervals $\mathcal{P}_i(u^-,u^+)$ is defined by:
\begin{equation}
\mathcal{P}_i(u^-,u^+) =\left\{ \begin{array}{ll}
(u^-_i,-u^-_i) & \textrm{if $0\in(u^-_i,u^+_i)$ and
$|u^-_i|\leqslant|u^+_i|$}\\
(-u^+_i,u^+_i) & \textrm{if $0\in(u^-_i,u^+_i)$ and $|u^-_i|>|u^+_i|$}\\
(0,u^+_{min}) & \textrm{if $(u^-_i,u^+_i)=(0,u^+_i)$},
\end{array} \right.
\end{equation}
where $u^+_{min}=\min\{u^+_i \mid i \textrm{ such that }
u^-_i=0\}$.
\end{definition}

\begin{theorem}\label{twoside-sat-dis}
Consider a graph $\mathcal{G}$ with dynamics
(\ref{closedloop-sat-disturb}).
Suppose that every edge allows bi-directional flow, i.e., $u^-_i<0<u^+_i, i=1,\ldots,m$.
Then for any in/outflow $\bar{d}$ for which
there exists $\bar{x}_c\in
\mathcal{P}(u^-,u^+)$
such that $E\bar{d}=B\bar{x}_c$, the
trajectories of (\ref{closedloop-sat-disturb}) will converge to
\begin{equation}
 \mathcal{E}_{\mathrm{tot}} = \{ (x,x_c) \mid \frac{\partial H}{\partial x}(x) =
\alpha \mathds{1}, \, \alpha \in \mathbb{R}, \, B\sat(-x_c\,;u^-,u^+)+E\bar{d} =
0
\, \}.
\end{equation}
if and only if the graph $\mathcal{G}$ is weakly connected.
\end{theorem}

\begin{proof}
By the matching condition $E\bar{d}=B\bar{x}_c$ and the identity
\begin{equation}\label{identity}
\sat(x-\eta\,;u^-,u^+)+\eta=\sat(x\,;u^-+\eta,u^++\eta),\,
\forall \eta\in\mathbb{R}^n,
\end{equation}
the system $(\ref{closedloop-sat-disturb})$ can be
written as
\begin{equation}
 \begin{aligned}
  \dot{x} & = B\sat(-B^T\frac{\partial H}{\partial
x}(x)-\tilde{x}_c\,;u^-+\bar{x}_c,u^++\bar{x}_c), \\
\dot{\tilde{x}}_c & = B^T\frac{\partial H}{\partial
x}(x),
 \end{aligned}
\end{equation}
where $\tilde{x}_c=x_c-\bar{x}_c$.
Since by construction $(u^-+\bar{x}_c)_i<0,\,i=1,\ldots,m,$
the proof now follows from Corollary $\ref{twoside-sat-cor}$.
\end{proof}
The following theorem covers the case that every edge is uni-directional.
\begin{theorem}\label{sat-dis}
Consider a network $\mathcal{G}$ with dynamics
(\ref{closedloop-sat-disturb}) with flow constraints such that $u^-_i=0,
i=1,\ldots,m$
(uni-directional flow). Then for any $u^+\in\mathbb{R}_+^m$ and
any in/outflow $\bar{d}$
for which there exists $\bar{x}_c \in
\mathcal{P}(0_m,u^+)$
such that $E\bar{d}=B\bar{x}_c$, the
trajectories of (\ref{closedloop-sat-disturb}) converge to
\begin{equation}
 \mathcal{E}_{\mathrm{tot}} = \{ (x,x_c) \mid \frac{\partial H}{\partial x}(x) =
\alpha \mathds{1}, \, \alpha \in \mathbb{R}, \, B\sat(-x_c\,;0_m,u^+)+E\bar{d} =
0
\, \},
\end{equation}
if and only if the graph in the (only) orientation
compatible with the flow constraints is {\it strongly connected} and {\it
balanced}.
\end{theorem}
In order to prove Theorem $\ref{sat-dis}$ we need the following lemma. Recall
that a directed graph is {\it balanced} if every vertex has in-degree
(number of incoming edges) equal to out-degree (number of outgoing edges).
Furthermore, we will say that two {\it cycles} of a graph are
{\it non-overlapping} if they do not have any edges in common.

\begin{lemma}\label{lemma}
A strongly connected graph is balanced if and only if it can
be covered by non-overlapping cycles.
\end{lemma}

\begin{proof}
{\it Sufficiency}: If a graph can be covered by
non-overlapping cycles, then every vertex necessarily has the same in-degree and
out-degree; so this graph is balanced.

{\it Necessity}: Since the graph is strongly connected, every two vertices can be
connected by a directed path, and the graph can be covered by cycles. Now suppose
that the graph can{\it not} be covered by non-overlapping cycles. We will show
that this implies that the graph is not balanced.

Let $k$ be the smallest number of cycles needed to cover the graph, and let
$\mathcal{T}=(C_1,C_2,\ldots,C_k)$ be a covering set of cycles. According to our
assumption, at least one edge of the graph is shared by two or more cycles in $\mathcal{T}$.
We claim that the set of shared edges can not contain any cycles. Indeed,
suppose there is one cycle, denoted as $\mathcal{D}$ (depicted in Fig
\ref{subgraph_TD}(a)), whose edges are all shared by elements of $\mathcal{T}$.
If $\mathcal{D} \in \mathcal{T}$, then obviously $\mathcal{T}$ is not a minimal
covering set, since by deleting the cycle $\mathcal{D}$ from $\mathcal{T}$ we
have a covering set of $k-1$ elements.

Thus $\mathcal{D} \notin \mathcal{T}$. It can be seen that the minimal number
$c$ of cycles in $\mathcal{T}$ which cover $\mathcal{D}$ twice is at least $4$.
Denote such a minimal set of $c$ cycles in $\mathcal{T}$ which cover
$\mathcal{D}$ by $\mathcal{T}_\mathcal{D}$. We will now show that by combining
these $c$ cycles with the cycle $\mathcal{D}$ there exist $3$ cycles in the
original graph $\mathcal{G}$ which cover the subgraph given by
$\mathcal{T}_\mathcal{D}$; thus reaching a contradiction with the minimality of
$\mathcal{T}$. The construction of these $3$ cycles is indicated in Figure
\ref{subgraph_TD}. Consider for simplicity the case that $4$ cycles in
$\mathcal{T}$, denoted by $C_1, C_2, C_3, C_4$ cover $\mathcal{D}$ twice.
Combining (depending on the orientation of the cycles) part of $C_1$ with part
of $C_3$, and part of $C_2$ with part of $C_4$ (see Figure \ref{subgraph_TD}),
we can define $2$ cycles which together with the cycle $\mathcal{D}$ yields a
set of $3$ cycles which cover the subgraph spanned by $C_1, C_2, C_3, C_4$.

In conclusion, there must exist at least one shared edge, say $(v_i,v_j)$, such that all edges with
tail-vertex $v_j$ are used only once in $\mathcal{T}$. But this implies that $v_j$ has larger
out-degree than in-degree, i.e., the graph is unbalanced.
\end{proof}

\begin{figure}
  \centering
  \subfigure[The cycle $\mathcal{D}$ split into two parts $\mathcal{D}_1, \mathcal{D}_2$.]{
    \label{fig:subfig:a} 
    \includegraphics[width=1.4in]{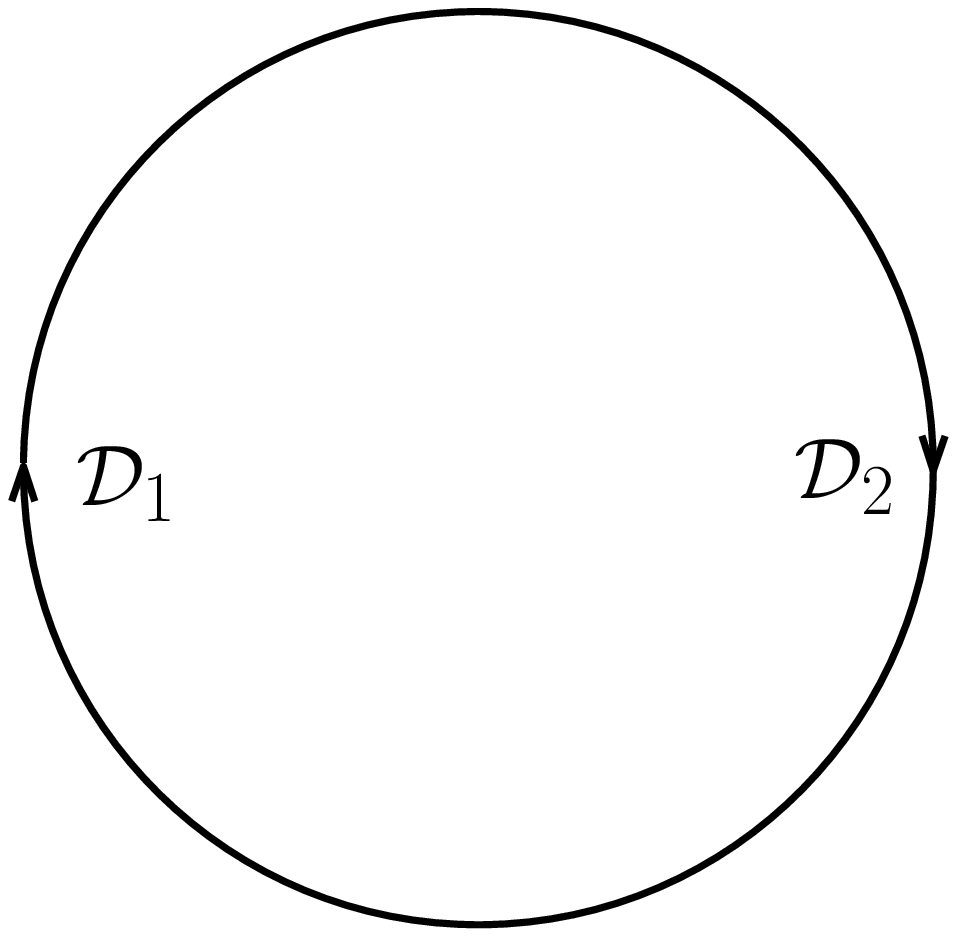}}
  \hspace{1in}
  \subfigure[The subgraph $\mathcal{T}_\mathcal{D}$]{
    \label{fig:subfig:b} 
    \includegraphics[width=2.7in]{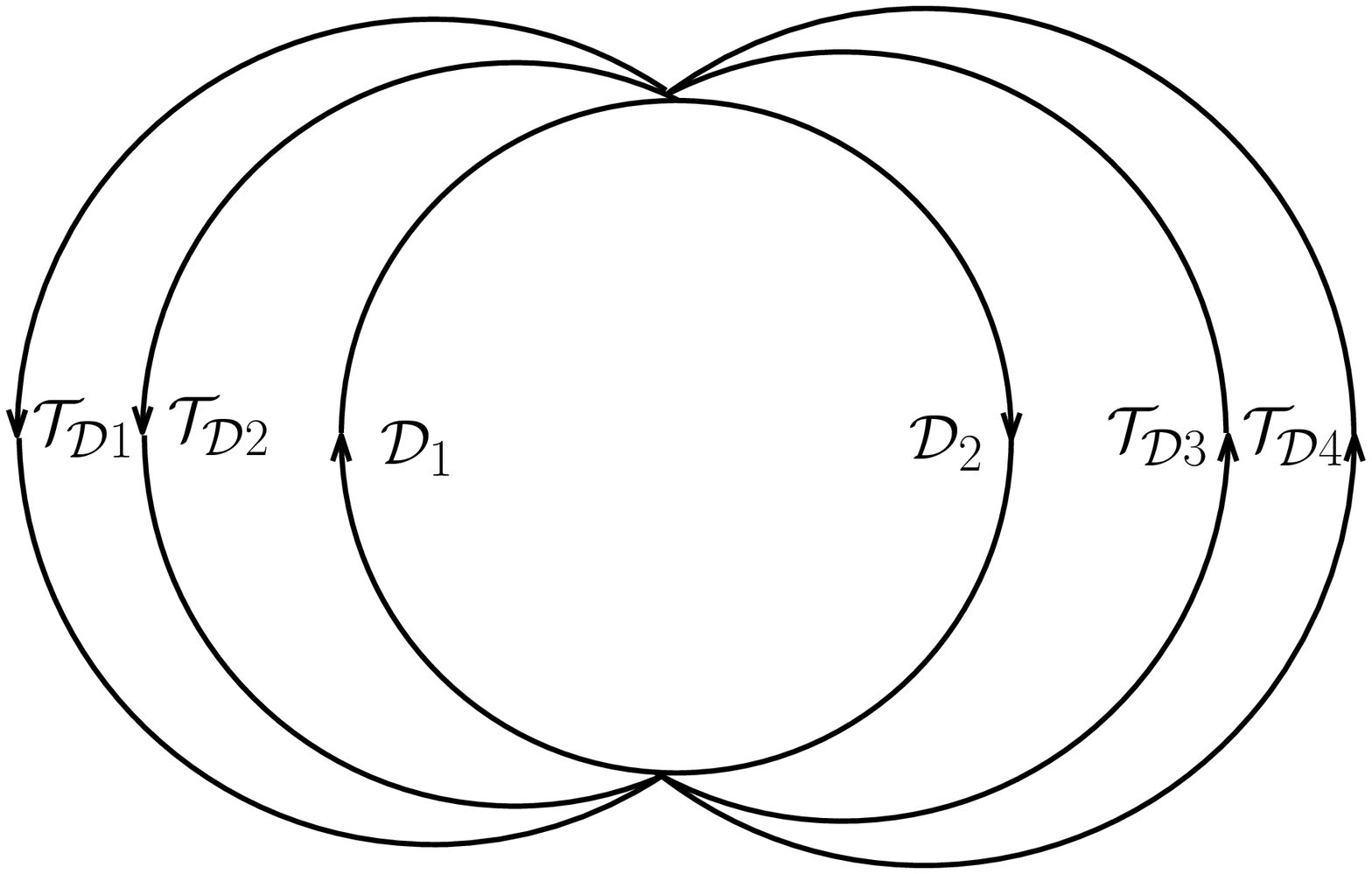}}
  \caption{(a). The cycle $\mathcal{D}$. We divide the edges of $\mathcal{D}$
into two disjoint sets: $\mathcal{D}_1$ contains the left part of $\mathcal{D}$
and $\mathcal{D}_2$ contains the rest. (b). The subgraph given by
$\mathcal{T}_{\mathcal{D}}$. Without $\mathcal{D}$, we need at least $4$ cycles to cover
$\mathcal{D}$ twice; these cycles are given as
$C_1=\mathcal{T}_{\mathcal{D}1}\cup\mathcal{D}_1$,$C_2=\mathcal{T}_{\mathcal
{D}2}\cup\mathcal{D}_1$,$C_3=\mathcal{T}_{\mathcal{D}3}\cup\mathcal{D}
_2$ and $C_4=\mathcal{T}_{\mathcal{D}4}\cup\mathcal{D}_2$. It follows that
$\mathcal{T}_{\mathcal{D}}$ is also covered by the $3$ cycles:
$\mathcal{D}$(clockwise),
$\mathcal{T}_{\mathcal{D}2} \cup \mathcal{T}_{\mathcal{D}3}$ (counterclockwise)
and
$\mathcal{T}_{\mathcal{D}1}\cup \mathcal{T}_{\mathcal{D}4}$ (counterclockwise).}
  \label{subgraph_TD} 
\end{figure}

\begin{proof}[Proof of Theorem $\ref{sat-dis}$]
{\it Sufficiency}:
By using $(\ref{identity})$ we rewrite the system as
 \begin{equation}
 \begin{aligned}
  \dot{x} & = B\sat\big(-B^T\frac{\partial H}{\partial
x}(x)-\tilde{x}_c\,;\bar{x}_c,u^++\bar{x}_c\big), \\
\dot{\tilde{x}}_c & = B^T\frac{\partial H}{\partial
x}(x),
 \end{aligned}
\end{equation}
where $\tilde{x}_c=x_c-\bar{x}_c$.
Consider now the Lyapunov function
\begin{equation}
 V(x,x_c)=\mathds{1}^{T} S\big(-B^{T}\frac{\partial H}{\partial
x}(x)-\tilde{x}_c\,;\bar{x}_c,u^++\bar{x}_c\big)+H(x).
\end{equation}
Similar to the proof of Theorem \ref{th:zerodisturbance}, if
a solution $(x(t),\tilde{x}_c(t))$ is in the largest invariant set $\mathcal{I}$
contained in $\{(x,\tilde{x}_c) \mid  \dot{V}=0\,\}$, then $x$ is a constant
vector,
denoted as $\nu$. Furthermore, $\mathcal{I}$ is given as
\begin{equation}
\begin{aligned}
 &\mathcal{I}=\{(\nu,\tilde{x}_c)\mid \tilde{x}_c=B^T\frac{\partial H}{\partial
x}(\nu)
t+\tilde{x}_c(0),\\
&B\sat\big(-B^T\frac{\partial H}{\partial
x}(\nu)-B^T\frac{\partial H}{\partial
x}(\nu)
t-\tilde{x}_c(0)\,;\bar{x}_c,u^++\bar{x}_c\big)=0, \forall t\geq 0\}.
\end{aligned}
\end{equation}
Suppose now that $B^T\frac{\partial H}{\partial
x}(\nu) \neq 0$. Then for $t$ large enough, we have
\begin{equation}\label{contradiction-directed-disturbance}
\begin{aligned}
0 &= \frac{\partial^T H}{\partial
x}(\nu) B\sat\big(-B^T\frac{\partial H}{\partial
x}(\nu)-B^T\frac{\partial H}{\partial
x}(\nu)
t-\tilde{x}_c(0)\,;\bar{x}_c,u^++\bar{x}_c\big) \\
& =
\sum_{i=1}^{m}\big(B^T\frac{\partial H}{\partial
x}(\nu)\big)_ic_i,
 \end{aligned}
\end{equation}
where
\begin{equation}
c_i  = \left\{ \begin{array}{ll}
\bar{x}_{ci} & \textrm{if $\big(B^T\frac{\partial
H}{\partial x}(\nu)\big)_i>0$},\\
u^+_i+\bar{x}_{ci} & \textrm{if $\big(B^T\frac{\partial
H}{\partial x}(\nu)\big)_i<0$}.
\end{array} \right.
\end{equation}
Since the graph is balanced we have $B\mathds{1}_m=0$, and thus
\begin{equation}
\sum_{i=1}^{m}\big(B^T\frac{\partial H}{\partial
x}(\nu)\big)_i=0.
\end{equation}
By the definition of the
permission set $\mathcal{P}(0_m,u^+)$, $0<\bar{x}_{ci}<u^+_j+\bar{x}_{cj}$ for
any $i,j=1,2,\ldots,m$, so
\begin{equation}
\sum_{i=1}^{m}\big(B^T\frac{\partial H}{\partial
x}(\nu)\big)_ic_i<0.
\end{equation}
This yields a contradiction. Hence $B^T\frac{\partial H}{\partial
x}(\nu)=0$ and therefore
\[
\mathcal{I}=\{(\nu,\tilde{x}_c)|\frac{\partial H}{\partial
x}(\nu)=c\mathds{1}_n,B\sat(-\tilde{x}_c\,;\bar{x}_c,u^++\bar{x}_c)=0\}
\]

\noindent
{\it Necessity}:
First, if the graph $\mathcal{G}$ is not strongly connected, by the same
argument in Theorem \ref{th:zerodisturbance}, it can be easily seen that
$\frac{\partial H}{\partial x}$ will not reach consensus.

Now we will show that if the
strongly connected network is unbalanced, then there exists a constraint interval $[0_m,u^+]$ and an in/outflow $\bar{d}$ for which there exists $\bar{x}_c\in \mathcal{P}(0_m,u^+)$ such that
$E\bar{d}=B\bar{x}_c$ while $\frac{\partial H}{\partial x}(x)$ is not converging to consensus.

For simplicity of exposition we shall take the constraint interval as $[0_m,\mathds{1}_m]$.

As in the proof of Lemma $\ref{lemma}$ we let $k$ be the minimal number of cycles to cover $\mathcal{G}$, and we let $\mathcal{T}=(C_1,\ldots,C_k)$ be a minimal covering set for $\mathcal{G}$. With some abuse of notation
\begin{equation}
 BC_i=0, i=1, \ldots,k
\end{equation}
where $C_i$ is the $m$-dimensional vector whose $j$-th component is
equal to the number of times the $j$-th edge appears in the cycle $C_i$.

In the following, we will prove that there exist $B^T\frac{\partial H}{\partial
x}(\nu)\neq0, \bar{x}_c\in\mathcal{P}(0_m,\mathds{1}_m)$,
$\tilde{x}_c(0)$ and $\lambda\in\mathbb{R}$, such that
\begin{equation}\label{constructed}
 \sat\big(-B^T\frac{\partial H}{\partial
x}(\nu)-B^T\frac{\partial H}{\partial
x}(\nu)
t-\tilde{x}_c(0)\,;\bar{x}_c,\mathds{1}_m+\bar{x}_c\big)=\lambda T,\, \forall
t\geq 0,
\end{equation}
where $\nu$ is the equilibrium value of $x$ as above and $T$ is the
$m$-dimensional vector whose $i$-th component is the number of cycles in
$\mathcal{T}$ which contains $i$-th edge. This implies that the
system has an equilibrium which
satisfies $\frac{\partial H}{\partial
x}(\nu)\notin\spa\{\mathds{1}_n\}$.

Consider as above a minimal covering set $\mathcal{T}=(C_1,\ldots,C_k)$ for
$\mathcal{G}$. Let $T_{\max}:= \{ T_i \mid i=1,2,\ldots,m\}$, and denote
$\mathcal{E}_1 = \{ i \mbox{-th edge} \mid T_i =T_{\max} \}$. Every cycle in
$\mathcal{T}$ has
at least one non-overlapped edge (see the proof of Lemma \ref{lemma}), and we
denote
by $\mathcal{E}_2$ the set of all the non-overlapped edges in the cycles in
$\mathcal{T}$ which contain at least one edge which is overlapped $l_{\max}$
times.

In the last step, we will make the flows through the edges in $\mathcal{E}_1$
reach the upper bounds of the constraint intervals, and the flows through the edges in $\mathcal{E}_2$ reach their lower
bounds. By taking
\begin{equation}\label{bounds}
\begin{cases}
\frac{\partial H}{\partial x}(\nu)_{j}<\frac{\partial H}{\partial x}(\nu)_{i}, &
(v_i,v_j)\in\mathcal{E}_1\\
\frac{\partial H}{\partial x}(\nu)_{j}>\frac{\partial H}{\partial x}(\nu)_{i}, &
(v_i,v_j)\in\mathcal{E}_2\\
\frac{\partial H}{\partial x}(\nu)_{j}=\frac{\partial H}{\partial x}(\nu)_{i}, &
\textrm{else}
\end{cases}
\end{equation}
for suitable $\bar{x}_c$ and
$\tilde{x}_c(0)$, it follows that (\ref{constructed}) holds. Indeed, in the set
$\mathcal{E}_1\cup\mathcal{E}_2$, the equation (\ref{constructed}) takes the form
\begin{equation}\label{e1+e2}
 \begin{aligned}
1+\bar{x}_{cq} & = \lambda T_q, q-\textrm{th edge belongs to
}\mathcal{E}_1\\
  \bar{x}_{cp} & = \lambda T_p, p-\textrm{th edge belongs to }\mathcal{E}_2
 \end{aligned}
\end{equation}
Now take $\lambda$ be such that $\frac{1}{l_{\max}}<\lambda<1$. Then (\ref{e1+e2}) contains $|\mathcal{E}_1|+|\mathcal{E}_2|$ equations and the same number of variables, and has a unique
solution such that
\begin{equation}
 \begin{aligned}
  0 &<\bar{x}_{cp}<1\\
0 &<\bar{x}_{cq}<1
 \end{aligned}
\end{equation}
Furthermore, pick $\tilde{x}_c(0)$ in the
third equation of (\ref{bounds}) as
\begin{equation}
 \tilde{x}_c(0)_r = -\lambda T_r, r-\textrm{th edge belongs to }\mathcal{E}\backslash
(\mathcal{E}_1\cup\mathcal{E}_2).
\end{equation}
Obviously, there exists $0<\bar{x}_{cr}<1$ such that
\begin{equation}
 \bar{x}_{cr}<-\tilde{x}_c(0)_r<1+\bar{x}_{cr}
\end{equation}
In conclusion, there exists an equilibrium $(\nu,\tilde{x}_c)$ that does not
satisfy
\\$B^T\frac{\partial H}{\partial x}(\nu) =0$, and thus $\frac{\partial
H}{\partial
x}(\nu)$ can not reach consensus.
\end{proof}

The above constructive proof is illustrated by the following example.

\begin{exmp}\label{example}
 Consider a directed graph in Figure
$\ref{figure}$ with dynamics given by system $(\ref{closedloop-sat-disturb})$
where $H(x)=\frac{1}{2}\|x\|^2$ and $[u^-,u^+]=[0_7,\mathds{1}_7]$
\begin{equation}\label{ex}
\begin{aligned}
\dot{x} & =  B\sat(-B^Tx-x_c\,;0_7,\mathds{1}_7)+E\bar{d},
\\[2mm]
\dot{x}_c & =  B^Tx.
\end{aligned}
\end{equation}
The purpose of this example is to show
there exist in/outflows $d$ satisfying the matching condition for which $x$ does
not converge to consensus.
\begin{figure}[ht]
\begin{center}
\begin{tikzpicture}
\tikzstyle{EdgeStyle}    = [thin,double= black,
                            double distance = 0.5pt]
\useasboundingbox (0,0) rectangle (4cm,4cm);
\tikzstyle{VertexStyle} = [shading         = ball,
                           ball color      = white!100!white,
                           minimum size = 20pt,%
                           inner sep       = 1pt,]
\Vertex[style={minimum
size=0.2cm,shape=circle},LabelOut=false,L=\hbox{$1$},x=1.5cm,y=3cm]{v1}
\Vertex[style={minimum
size=0.2cm,shape=circle},LabelOut=false,L=\hbox{$2$},x=0cm,y=1.5cm]{v2}
\Vertex[style={minimum
size=0.2cm,shape=circle},LabelOut=false,L=\hbox{$3$},x=1.5cm,y=0cm]{v3}
\Vertex[style={minimum
size=0.2cm,shape=circle},LabelOut=false,L=\hbox{$4$},x=3cm,y=1.5cm]{v4}
\Vertex[style={minimum
size=0.2cm,shape=circle},LabelOut=false,L=\hbox{$5$},x=0cm,y=3cm]{v5}
\draw
(v1) edge[->,>=angle 90,thin,double= black,double distance = 0.5pt]
node[right]{$e_1$} (v2)
(v2) edge[->,>=angle 90,thin,double= black,double distance = 0.5pt]
node[left]{$e_2$} (v3)
(v3) edge[->,>=angle 90,thin,double= black,double distance = 0.5pt]
node[right]{$e_3$} (v1)
(v1) edge[->,>=angle 90,thin,double= black,double distance = 0.5pt]
node[right]{$e_4$} (v4)
(v4) edge[->,>=angle 90,thin,double= black,double distance = 0.5pt]
node[right]{$e_5$} (v3)
(v1) edge[->,>=angle 90,thin,double= black,double distance = 0.5pt]
node[above]{$e_6$} (v5)
(v5) edge[->,>=angle 90,thin,double= black,double distance = 0.5pt]
node[left]{$e_7$} (v2);
\end{tikzpicture}
\caption{Network of Example $\ref{example}$}\label{figure}
\end{center}
\end{figure}
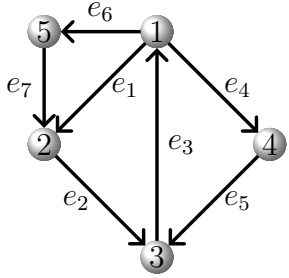
By taking $E\bar{d}=B\bar{x}_c$ where $\bar{x}_c=\frac{1}{2}\mathds{1}_7,
x(0)=(3,7 5,1,4)^T$ and
$\tilde{x}_c(0)=(1,-1,-1,1,1,1,1)^T$, the state $x$ in system $(\ref{ex})$ will converge
to $\nu$ with $\nu_2=\nu_3>\nu_5>\nu_4>\nu_1$ and $\nu_4<\nu_1$ as can
be seen from the numerical simulation in Figure $\ref{simulation}$.
\begin{figure}[htbp]
 \centering
 \includegraphics[height=5.8cm,width=7.8cm]{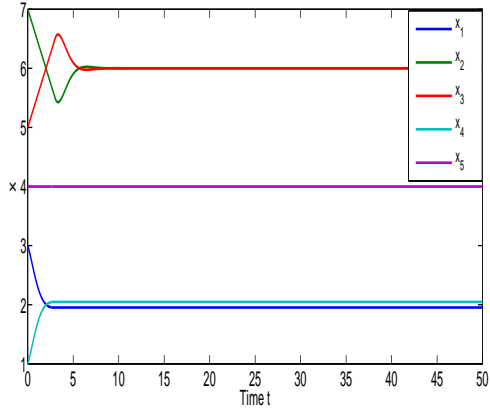}
 \caption{The time-evolutions $x_1(t),x_2(t),x_3(t),x_4(t),x_5(t)$ of the
system (\ref{ex}).}
 \label{simulation}
 \end{figure}

The same result can be deduced from the following analysis. In Figure
$\ref{figure}$, the smallest number of cycles to cover the whole graph is $3$;
one option being $(e_6,e_7,e_2,e_3)$, $(e_1,e_2,e_3)$, $(e_3,e_4,e_5)$. So
$BT=0$
where
\begin{equation}
T=(1,2,3,1,1,1,1)^T
\end{equation}
In this case $\mathcal{E}_1=\{e_3\},\mathcal{E}_2=\{e_1,e_4,e_5,e_6,e_7\}$. By
setting
$\nu_2=\nu_3>\nu_5>\nu_4>\nu_1$, the flow in
$e_3$
reaches the upper bound, while the flows in $e_1,e_4,e_5,e_6,e_7$ reach the
lower bounds, i.e.
\begin{equation}
\sat(-B^T\nu-B^T\nu
t-\tilde{x}_c(0)\,;\bar{x}_c,\bar{x}_c+\mathds{1})=\frac{1}{2}T,\,\forall t>0.
\end{equation}
Thus there exists an equilibrium $\nu$ satisfying $B^T\nu\neq0$.
\end{exmp}

\section{Conclusions}
We have discussed a basic model of dynamical distribution networks where the
flows through the edges are generated by distributed PI controllers. The
resulting system can be naturally modeled as a port-Hamiltonian system, enabling
the easy derivation of sufficient and necessary conditions for the convergence
of the state variables to load balancing (consensus). The main part of this
paper focusses on the case where flow constraints are present. A key ingredient
in this analysis is the construction of a $C^1$ Lyapunov function. We
distinguish between the case that the flow constraints corresponding to all the
edges allow for bi-directional flow and the case that all the edges only allow
for uni-directional flow. For both cases we have derived necessary and
sufficient conditions for asymptotic load balancing based on the structure of
the graph.

An obvious open problem is the extension of our results to the general case where some of the edges allow bi-directional flow and others only uni-directional flow. This is currently under investigation.
Many other questions can be addressed in this framework. For example, what is happening if
the in/outflows are not assumed to be constant, but are e.g. periodic
functions of time; see already \cite{depersis}.
Furthermore, the use of constrained PI-controllers may
suggest a fruitful connection to anti-windup control ideas.






\end{document}